\documentclass[a4paper,11pt,reqno,twoside]{amsart}

\usepackage{amsmath}
\usepackage{amsfonts}
\usepackage{amssymb}
\usepackage{amsthm}
\usepackage{color}
\usepackage{ifpdf}
\usepackage{array}
\usepackage{url}
\usepackage{multirow,bigdelim}
\usepackage{ascmac}
\usepackage{ulem}

\numberwithin{equation}{section}

\newtheorem{theorem}{Theorem}[section]

\newtheorem{proposition}{Proposition}[section]
\newtheorem{corollary}{Corollary}[section]

\theoremstyle{definition}
\newtheorem{remark}{Remark}[section]

\newcommand*{\C}{\mathbb{C}}
\newcommand*{\R}{\mathbb{R}}
\newcommand*{\Z}{\mathbb{Z}}
\newcommand*{\T}{\mathbb{T}}

\newcommand{\comment}[1]{}
\title[$M$-functions and screw functions]%
      {$M$-functions and screw functions : \\ 
applications to Goldbach's problem and \\ 
zeros of the Riemann zeta-function} 
\author[K. Matsumoto]{Kohji Matsumoto}
\author[M. Suzuki]{Masatoshi Suzuki}
\date{Version of \today}
\AtBeginDocument{%
\begin{abstract}
We study the $M$-functions, 
which describe the limit theorem for the value-distributions of 
the secondary main terms in the asymptotic formulas for the summatory functions 
of the Goldbach counting function. 
One of the new aspects is a sufficient condition for the Riemann hypothesis 
provided by some formulas of the $M$-functions, 
which was a necessary condition in previous work. 
The other new aspect is the relation between the secondary main terms and the screw functions, 
which provides another necessary and sufficient condition for the Riemann hypothesis. 
We study such $M$-functions and screw functions in generalized settings 
by axiomatizing them.
\end{abstract}
\maketitle
}
\begin{document}

%
\section{Introduction} 
%

The absolutely convergent series 
\begin{equation} \label{f_101}
H(X) := \sum_{\rho} \frac{X^{\rho-1/2}}{\rho(\rho+1)}, \quad X \geq 1
\end{equation}
over nontrivial zeros $\rho$ of the Riemann zeta function $\zeta(s)$ 
was studied by Fujii in his series of papers \cite{Fu1, Fu2, Fu3},   
motivated in part by its contribution to the study of the summatory function
\[
\sum_{n \leq X}
\left(
\sum_{m+k=n} \Lambda(m)\Lambda(k) \right) 
\]
related to the Goldbach conjecture, 
where $\Lambda(n)$ is the von Mangoldt function defined by 
$\Lambda(n)=\log p$ if $n=p^k$ for some prime number $p$ with $k \in \Z_{>0}$ 
and $\Lambda(n)=0$ otherwise. 
The series $H(X)$ is real-valued, 
since the set of nontrivial zeros of $\zeta(s)$ is closed under complex conjugation. 
The asymptotic formula 
\begin{equation} \label{f_102}
\sum_{n \leq X}
\left(
\sum_{m+k=n} \Lambda(m)\Lambda(k) \right) 
= \frac{1}{2}X^2 - 2X^{3/2}\,H(X) + R(X) 
\end{equation}
with the estimate $R(X)=O((X\log X)^{4/3})$ 
was proved first by Fujii \cite{Fu2} under the Riemann hypothesis, 
which asserts that $\Re(\rho)=1/2$ for all nontrivial zeros $\rho$. 
After that, $R(X)=O(X^{1+\varepsilon})$ 
was conjectured by Egami and Matsumoto \cite{EgMa07}, 
and $R(X)=O(X(\log X)^5)$ was proven by Bhowmik and Schlage-Puchta \cite{BS10} 
under the Riemann hypothesis. 
The conditional error term has been improved to $O(X(\log X)^3)$ 
by Languasco and Zaccagnini~\cite{LZ12} 
and later by Goldston and Yang~\cite{GY17} by a different method. 
These two estimates are close to the omega result $\Omega(X\log\log X)$ 
obtained in \cite{BS10}. 
There are other interesting studies on 
the relation between formula \eqref{f_102} 
and zeros of the Riemann zeta function, 
such as Bhowmik and Ruzsa \cite{BR18}, 
Billington, Cheng, Schettler, and Suriajaya
\cite{BCSS23}, and references therein, 
but we will not discuss them in detail here, 
since the subject of this paper is the value-distribution of 
somewhat general sums including \eqref{f_101}. 

The series $H(X)$ also appears in the average of the oscillatory term of 
the Chebyshev function as 
\begin{equation} \label{f_103}
\aligned 
\int_{0}^{X} \left[ \sum_{n \leq y} \Lambda(n) - y \right] dy
& = - X^{3/2}H(X)
-\frac{\zeta'}{\zeta}(0)X \\
& \qquad  +\frac{\zeta'}{\zeta}(-1)- X \sum_{n=1}^{\infty} \frac{X^{-2n}}{2n(2n-1)}
\endaligned 
\end{equation}
for $X \geq 1$ (\cite[p. 249]{Fu2}). 

Concerning the value of $H(X)$, 
first, it is bounded under the Riemann hypothesis, 
more precisely $|H(X)/2|<0.023059$  (\cite[(1.7)]{MoTr22}), 
due to the absolute convergence of the series. 
Fujii \cite{Fu3} proved that $H(X)/2 > 0.012$ for infinitely many $X$ 
and $H(X)/2 < - 0.012$ for infinitely many $X$. 
Mossinghoff and Trudgian~\cite{MoTr22} improved 
these inequalities into $H(X)/2 > 0.021030$ and $H(X)/2 < - 0.022978$. 
They derived the results by assuming only the Riemann hypothesis. 
The first author \cite{Ma21} proved more detailed results for the value-distribution of $H(X)$ 
further assuming the linear independence over rationals 
for the set of positive imaginary parts of the nontrivial zeros. 

We first highlight a few results which are later proved in a more general context. 

\begin{theorem}
\label{thm_1_1} 
We assume the Riemann hypothesis and 
the linear independence over rationals 
for the set of positive imaginary parts of the nontrivial zeros of the Riemann zeta function. 
Then, there exists an explicitly constructible density function 
$M_H: \R \to \R_{\geq 0}$, for which
\begin{equation} \label{f_104}
\lim_{T \to \infty} \frac{1}{T}\int_{0}^{T} 
\Phi(H(e^t)) \, dt = \frac{1}{\sqrt{2\pi}} \int_\R M_H(u) \Phi(u) \, du
\end{equation}
holds for any test function $\Phi: \R \to \C$ which is locally Riemann integrable. 
The function $M_H(u)$ is continuous, nonnegative, compactly supported, and
$\frac{1}{\sqrt{2\pi}} \int_\R M_H(u)\, du = 1$. 
\end{theorem}
We immediately obtain this result 
as a corollary of \cite[Theorem 2.3]{Ma21} 
(see also the proof of Corollary \ref{cor_2_1} below).  
We refer to the function $M_H$ appearing in formula \eqref{f_104} 
as the {\it $M$-function}, following \cite{I08} and \cite {Ma21}.
Note that the Riemann hypothesis is not assumed in \cite[Theorem 2.3]{Ma21}, 
because it studies the series 
\[
\sum_{\rho} \frac{X^{i\Im(\rho)}}{\rho(\rho+1)}, 
\]
not the series \eqref{f_101}. 

Assuming the Riemann hypothesis, $|X^{\rho-1/2}|=1$ in \eqref{f_101}. 
Thus, the absolute convergence of $H(X)$ implies its boundedness. 
The compactness of the support of $M_H(u)$ in formula \eqref{f_104} 
is a consequence of the boundedness of $H(X)$ under the Riemann hypothesis 
(see the proof of Theorem \ref{thm_1_1} after Corollary \ref{cor_2_1}). 
As the converse of Theorem \ref{thm_1_1}, 
we observe that the boundedness of $H(X)$ follows from formula \eqref{f_104}. 
On the other hand, as will be shown in the more general setting in Section \ref{section_3}, 
boundedness of $H(X)$ implies that all $i(\rho-1/2)$ are real (cf. Corollary \ref{cor_3_1}), 
that is, the Riemann hypothesis holds. 
Therefore, we obtain the following: 

\begin{theorem}
\label{thm_1_2} 
We assume that the formula \eqref{f_104} 
holds for a compactly supported continuous function $M_H: \R \to \R$ 
and for any test function $\Phi: \R \to \C$ which is locally Riemann integrable. 
Then $H(X)$ is bounded, and hence the Riemann hypothesis holds. 
\end{theorem}
This is a special case of Corollary~\ref{cor_3_2}.
In \eqref{f_104}, the function $M_H$ is a nonnegative function, 
but it is not assumed to be nonnegative in Theorem \ref{thm_1_2}
because it is not necessary for the proof.
\medskip

Next, we consider the following variant of \eqref{f_101} 
\begin{equation} \label{f_105}
H_1(X) :=
 \sum_{\rho}  \frac{X^{\rho-1/2}}{\rho(1-\rho)}, \quad X \geq 1,
\end{equation}
which not only has a value-distribution similar to \eqref{f_101}, 
but also has richer properties. 
We find that asymptotic formula \eqref{f_102} implies 
\begin{equation} \label{f_106}
\sum_{n \leq X} \frac{1}{n^2} \left(
\sum_{m+k=n} \Lambda(m)\Lambda(k) \right) 
 = \log X + c_2 + \frac{2}{\sqrt{X}} \, H_1(X) + E(X),
\end{equation}
where 
\begin{equation} \label{f_107}
c_2 = \lim_{X \to \infty}\left[\sum_{n \leq X}  \frac{1}{n^2} \left(
\sum_{m+k=n} \Lambda(m)\Lambda(k) \right)  - \log X \right]
\end{equation}
and $E(X)$ is the error term estimated as 
$E(X) \ll X^{-2}R(X) + \int_{X}^{\infty} y^{-3}R(y)dy$ 
using partial summation. 
Conversely, formula \eqref{f_102} 
with the estimate $R(X) \ll X^2E(X) + \int_{1}^{X} y E(y)dy$ 
can be derived from \eqref{f_106} by partial summation 
(see Section \ref{section_7} for details).
In this sense, the contributions of series $H(X)$ and $H_1(X)$ 
to the Goldbach problem are equivalent. 
Furthermore, the argument in \cite{Ma21} can be applied to $H_1(X)$, 
and therefore results similar to Theorems \ref{thm_1_1} and \ref{thm_1_2} hold for $H_1(X)$
(Corollary \ref{cor_2_1} and Theorem \ref{thm_3_1}  below). 

The series $H_1(X)$ also appears in the study of 
Euler's totient function as follows 
\[
\aligned 
\limsup_{n \to \infty} \left(
\frac{n}{\varphi(n)} - e^{C_0}\log\log n
\right) \sqrt{n} 
& = e^{C_0}
(2 + \limsup_{X \to \infty}H_1(X)) \\
& = 
e^{C_0}(2 + H_1(1))
\endaligned 
\]
(Nicolas \cite[Theorem 1.1, p. 319]{Ni12}), 
where $C_0$ is the Euler-Mascheroni constant. 
From the definition, 
the constant $c_2$ in \eqref{f_107} may be considered an analog of $C_0$.
\medskip

To discuss the advantages of $H_1(X)$, we briefly review the class of screw functions 
according to Kre\u{\i}n and Langer \cite{KrLa14}.
A continuous function $g(t)$ on $\R$ 
is called a screw function on $\R$ 
if it satisfies $g(-t)=\overline{g(t)}$ for all $t \in \R$ 
and the kernel 
\[
G_g(t,u):=g(t-u)-g(t)-g(-u)+g(0)
\] 
is nonnegative definite on $\R\times\R$, that is, 
$\sum_{i,j=1}^{n} G_g(t_i,t_j) \,  \xi_i \overline{\xi_j}  \geq 0$ 
for any $n \in \Z_{>0}$, $t_i \in \R$, and $\xi_i \in \C$. 
The class of screw functions was introduced 
as a natural generalization of positive definite functions and 
is an interesting subject related to various topics in analysis 
as explained in \cite[Section 1]{KrLa14}. 
That class has recently been applied to the study of the zeta function 
by the second author \cite{Su23}.
The series $H_1(X)$ is related to the theory of screw functions as follows.

\begin{theorem}
\label{thm_1_3} 
The function 
\[
g_{H_1}(t) := H_1(e^t) - H_1(1)
\]
is a screw function on $\R$ 
if and only if the Riemann hypothesis holds. 
\end{theorem}
This is a special case of Corollary~\ref{cor_4_1}. 
It shows a remarkable property of $H_1(X)$ that $H(X)$ does not have. 
In fact, any relation between the latter and a screw function is not known. 
Further, the $M$-function of $H_1(X)$ relates to an infinitely divisible distribution 
via the attached screw function. 
A distribution $\mu$ on $\R$ is called infinitely divisible 
if there exists a distribution $\mu_n$ on $\R$ 
such that 
$\mu=\mu_n \ast \dots \ast \mu_n$ ($n$-fold) for every positive integer $n$. 
We find that if $g(t)$ is a screw function, 
then $\exp(g(t))$ is the characteristic function of an infinitely divisible distribution 
by \cite[Theorem 5.1]{KrLa14}  and \cite[Theorem 8.1 and Remark 8.4]{Sa99} 
(see also \cite{NaSu23} and Section \ref{section_5}). 
Therefore, by Theorem \ref{thm_1_3}, 
there exists an infinitely divisible distribution corresponding to $g_{H_1}(t)$ 
under the Riemann hypothesis.

\begin{theorem}
\label{thm_1_4} 
We assume the Riemann hypothesis and 
the linear independence over rationals 
for the set of positive imaginary parts of the nontrivial zeros of $\zeta(s)$. 
Let $M_{H_1}(w)$ be the $M$-function in 
the analog of Theorem \ref{thm_1_1} for $H_1(X)$ 
(Corollary \ref{cor_2_1} and Theorem \ref{thm_3_1}  below).
For $y>0$, let $\mu_y(x)$ be the infinitely divisible distribution on $\R$ 
whose characteristic function is $\exp(yg_{H_1}(t))$: 
\[
\exp(yg_{H_1}(t)) =  \int_{-\infty}^{\infty} e^{itx} \mu_y(dx). 
\]
Then the value of the point mass of $\mu_y(x)$ at the origin 
is given by the $M$-function as follows: 
\begin{equation} \label{f_108}
\mu_y(\{0\}) 
= e^{-yH_1(1)} \widetilde{M_{H_1}}(-iy) 
= e^{-yH_1(1)} \prod_{\gamma>0}
J_0\left(\frac{2iym_\gamma}{1/4+\gamma^2} \right), 
\end{equation}
where 
\[
\widetilde{M_{H_1}}(z) = \frac{1}{\sqrt{2\pi}} \int_\R M_{H_1}(u) e^{izu} \, du \quad (z \in \C), 
\] 
the product $\prod_{\gamma>0}$ ranges over all positive imaginary parts $\gamma$ 
of nontrivial zeros of $\zeta(s)$ without multiplicity, 
$m_\gamma$ is the multiplicity of the nontrivial zero $1/2+i\gamma$, 
and $J_0(z)$ is the Bessel function of the first kind of order zero. 
\end{theorem}
This is a special case of Theorem~\ref{thm_5_1}. 
We do not know what equation \eqref{f_108} can be applied to, 
but it is interesting in its own right, 
because, in general, 
it is difficult to calculate values of the corresponding infinitely divisible distribution 
for a given screw function (or a L{\'{e}}vy measure). 
\medskip

In the following sections, we prove Theorems \ref{thm_1_1} to \ref{thm_1_4} in more general settings 
including the series
\begin{equation} \label{f_109}
H_\ell(X) :=
 \sum_{\rho}  \frac{X^{\rho-1/2}}{(\rho-\ell)((1-\rho)-\ell)}, \quad X \geq 1, \quad \ell \in \R, 
\end{equation}
where the sum is taken with multiplicity. 
The series $H_\ell(X)$ is nothing but \eqref{f_105} when $\ell=1$, 
and is real-valued 
by the same reasoning
as \eqref{f_101}.
The difference $H_{1/2}(e^t)-H_{1/2}(1)$ becomes the screw function 
of $\zeta(s)$ studied in \cite{Su23} 
(see the comments after Proposition \ref{prop_6_3}). 
The reason why we proceed with the discussion in general settings 
is to show that the value distributions of absolutely convergent oscillatory sums 
such as \eqref{f_101}, \eqref{f_105}, and \eqref{f_109} 
can be discussed regardless of the specific form of the coefficients. 
As a result, such a theory can not only be applied to the sums 
replacing the nontrivial zeros of the Riemann zeta-function with 
those of Dedekind zeta-functions or automorphic $L$-functions 
in \eqref{f_101}, \eqref{f_105}, and \eqref{f_109}, 
but also to sums of different analytic nature, such as
\begin{equation} \label{f_110}
\sum_\rho \frac{\Gamma((1-\rho)/2)}{\zeta'(\rho)} \, X^{\rho-\frac{1}{2}}, 
\end{equation}
which appears in a formula of Ramanujan, 
and whose convergence is more delicate compared to \eqref{f_101} and \eqref{f_105}. 
The conditional convergence of \eqref{f_110}  is discussed in Titchmarsh \cite[pp. 219--220]{Tit86}.
Chirre and Gonek~\cite{ChGo23} proved that \eqref{f_110} converges absolutely 
under what they named the ``Weak Mertens Hypothesis'', 
and studied its value distribution under 
the linear independence over rationals 
for the set of positive imaginary parts of the nontrivial zeros of $\zeta(s)$. 
\medskip

The results mentioned in the introduction are proved in the following sections as special cases of the general theorems. 
In Sections \ref{section_2} and \ref{section_3} we study $M$-functions. 
In Section \ref{section_2}, 
we set up an axiomatic framework that includes $H(X)$ and $H_\ell(X)$ and 
prove formulas for $M$-functions (Theorem \ref{thm_2_1}) as a generalization of Theorem \ref{thm_1_1}. 
In Section \ref{section_3}, we prove Theorem \ref{thm_3_1}, which includes Theorem \ref{thm_1_2} 
and can be viewed as the converse of Theorem \ref{thm_2_1}, with one additional condition to the setting in Section \ref{section_2}. 
In Sections \ref{section_4} and \ref{section_5}, 
we study screw functions by adding a few conditions to the axioms in Sections \ref{section_2} and \ref{section_3}. 
Theorem \ref{thm_1_3} is proven as a special case of 
Corollary \ref{cor_4_1} 
in Section \ref{section_4}. 
Theorem \ref{thm_1_4} is proven as a special case of Theorem \ref{thm_5_1} in Section \ref{section_5}. 
Furthermore, 
we additionally provide explicit formulas for $H(X)$ and $H_\ell(X)$ 
that do not include nontrivial zeros of $\zeta(s)$ 
in their expressions in Section \ref{section_6}. 
Finally, we explain the equivalence of \eqref{f_102} and \eqref{f_106} under the Riemann hypothesis 
in Section \ref{section_7}.

\section{The axiomatic framework} \label{section_2}

\subsection{Generalization of Theorem \ref{thm_1_1}} \label{sec_2_1}

We prove Theorem \ref{thm_1_1} as a special case of the following general cases.  
Let $\Pi=(\Omega,a)$ be a pair of 
a countable set of nonzero complex numbers $\Omega$ with $\#\Omega\geq 5$  
and a function $a: \Omega \to \C\setminus\{0\}$ satisfying the following two conditions: 
\begin{enumerate}
\item[(M1)] $\sum_{\omega \in \Omega}|a(\omega)|$ converges.  
\item[(M2)]
There exists $c \geq 0$ such that $|\Im(\omega)| \leq c$ for every $\omega \in \Omega$. 
\end{enumerate}
The reason why zero is excluded from $\Omega$ is that 
it is meaningless in \eqref{f_301} below, 
and it is inconvenient when considering the linear independence of $\Omega$. 
Also, the reason why the function $a$ is assumed to be nonzero 
is that the proof of Proposition \ref{prop_4_1} below works.

Clearly, $\Omega \subset \R$ implies (M2). 
Furthermore, we denote by ${\rm LIC}(\Omega)$ 
the assertion that the set $\Omega$ is linearly independent over the rationals.  
The importance of linear independence in the theory of value-distribution of $\zeta(s)$ 
was probably first pointed out by Wintner~\cite{Win38}. 

The pair of the set
\begin{equation} \label{f_201}
\Omega_\zeta^{+}:=\{ \omega= i(\rho-1/2)~|~\zeta(\rho)=0,~
0 < \Re(\rho) < 1,~\Im(\rho)>0\}, 
\quad 
\end{equation}
($\rho=1/2-i\omega$, $\omega \in \Omega \subset \C$)  
and either 
\[
\aligned 
a_{H}(\omega):&= \frac{2m_\omega}{(1/2-i\omega)(3/2-i\omega)}
\quad \text{or} \\ 
a_{H_\ell}(\omega):&= \frac{2m_\omega}{(1/2-i\omega-\ell)(1/2+i\omega-\ell)}
\endaligned 
\] 
is an example of pairs satisfying (M1) and (M2), 
where $m_\omega$ is  the multiplicity of 
the nontrivial zero $\rho=1/2-i\omega$ of $\zeta(s)$. 
Note that the claim ${\rm LIC}(\Omega_\zeta^+)$ is different from 
the usual linear independence of the nontrivial zeros of $\zeta(s)$, 
which implies the simplicity of the zeros. 
As can be seen from the proof of \cite[Proposition 4.1]{Ma21} 
together with \cite[(3.1)]{Ma21}, 
LIC in \cite[Theorem 2.3]{Ma21} is used in the latter sense. 
On the other hand, 
${\rm LIC}(\Omega_\zeta^+)$ does not imply the simplicity of nontrivial zeros, 
since it excludes information about the multiplicities of the nontrivial zeros. 

For a pair satisfying (M1) and (M2), we define 
\begin{equation} \label{f_202}
f_\Pi(t) := \sum_{\omega \in \Omega}a(\omega)e^{-it\omega}. 
\end{equation}

\begin{theorem} \label{thm_2_1} 
Let $\Pi=(\Omega,a)$ be a pair satisfying (M1). 
We assume $\Omega \subset \R$ 
and ${\rm LIC}(\Omega)$. 
Then, there exists a $M$-function $M_\Pi: \C \to \R_{\geq 0}$, for which
\begin{equation} \label{f_203}
\lim_{T \to \infty} \frac{1}{T}\int_{0}^{T} 
\Phi(f_\Pi(t)) \, dt =  \int_\C M_\Pi(w) \Phi(w) \, |dw|
\end{equation}
holds for any test function $\Phi: \C \to \C$ which is locally Riemann integrable, 
where $|dw|=dudv/(2\pi)$ for $w=u+iv$. 
The function $M_\Pi(w)$ is explicitly constructible, 
continuous, nonnegative, compactly supported, and
\[
\int_\C M_\Pi(w)\, |dw| = 1. 
\]
More precisely, 
\begin{equation} \label{eq_0903_1}
{\rm supp}\,M_\Pi \subseteq \left\{ w \in \C \, :\, |w| \leq \sum_{\omega \in \Omega}|a(\omega)| \right\}.
\end{equation}
Equality holds in \eqref{eq_0903_1} if $\Pi$ satisfies 
\begin{equation} \label{eq_0903_2}
\max_{\omega \in \Omega}\,|a(\omega)|<\frac{1}{2}\sum_{\omega \in \Omega}|a(\omega)|. 
\end{equation}
\end{theorem}

Before proving Theorem \ref{thm_2_1}, 
we show that Theorem \ref{thm_1_1} is proved by the following corollary. 

\begin{corollary} \label{cor_2_1}
Under the same assumptions in Theorem \ref{thm_2_1}, 
we define 
\begin{equation} \label{f_204}
M_\Pi^\Re(u) := \frac{1}{\sqrt{2\pi}} \int_\R M_\Pi(u+iv) \, dv \quad (u \in \R)
\end{equation}
using the $M$-function $M_\Pi(w)$ in \eqref{f_203}. 
Then 
\begin{equation} \label{f_205}
\lim_{T \to \infty} \frac{1}{T}\int_{0}^{T} 
\Phi(\Re(f_\Pi(t))) \, dt =  \int_\R M_\Pi^\Re(u) \Phi(u) \, |du|
\end{equation}
holds for any test function $\Phi: \R \to \C$ which is locally Riemann integrable, 
where $|du|=du/\sqrt{2\pi}$. 
The function $M_\Pi^\Re(u)$ is continuous, nonnegative, compactly supported, and
\[
\int_\R M_\Pi^\Re(u)\, |du| = 1. 
\]
More precisely, 
\begin{equation} \label{eq_0809_1}
{\rm supp}\,M_\Pi^\Re \subseteq
\left\{ u \in \R \, :\, |u| \leq \sum_{\omega \in \Omega}|a(\omega)| \right\}, 
\end{equation} 
and the equality holds in \eqref{eq_0809_1} if $\Pi$ satisfies \eqref{eq_0903_2}. 
\end{corollary}
\begin{proof}
Applying \eqref{f_203} 
to $\Phi(w)=\psi_z(w)=\exp(i\Re(\bar{z}w))$ for $z \in \R$,  
\[
\aligned 
\lim_{T \to \infty} & \frac{1}{T}\int_{0}^{T} 
\exp\left(iz \Re(f_\Pi(t))\right) \, dt \\
&= 
\frac{1}{\sqrt{2\pi}}\int_\R 
\left[
\frac{1}{\sqrt{2\pi}}\int_\R M_\Pi(u+iv) dv \right] \exp(iz u)\, du. 
\endaligned 
\]
This implies \eqref{f_205} for $\Phi(w)=\psi_z(w)$. 
Then, formula \eqref{f_205} about general locally Riemann integrable functions $\Phi$ 
follows from arguments similar to those in \cite[Section 6]{Ma21}. 
\end{proof}

Theorem \ref{thm_1_1} is obtained as $M_H(u)=M_\Pi^\Re(u)$ 
by applying Corollary \ref{cor_2_1} to $\Pi=(\Omega_\zeta^+, a_H)$, since 
\[
H(e^t)=\Re(f_\Pi(t)) 
\quad \text{by} \quad 
f_\Pi(t) = 2 \sum_{\Im(\rho)>0}\frac{m_\rho}{\rho(\rho+1)} \, e^{t(\rho-1/2)}, 
\]
where $m_\rho:=m_\omega$ if $\rho=1/2-i\omega$ (cf. \eqref{f_201}). 
In particular, 
the compactness of the support of $M_H(u)$ follows from 
\eqref{eq_0809_1}, i.e., from the boundedness of $H(X)$. 
The analog of Theorem \ref{thm_1_1} for $H_\ell(X)$ ($\ell \in \R$)
is obtained 
by applying Corollary \ref{cor_2_1} to $\Pi=(\Omega_\zeta^+, a_{H_\ell})$ as well. 
See Section~\ref{section_2_3} for the confirmation that 
condition~\eqref{eq_0903_2} is satisfied in these cases. 
If we take the pair of 
\[
\Omega_\zeta^{\Im, +}:=\{ -\Im(\rho)~|~\zeta(\rho)=0,~0 < \Re(\rho) < 1,~\Im(\rho)>0\}
\]
and $a_H$, 
the analog of Theorem \ref{thm_1_1} for $\Psi(X)$ in \cite[(1.2)]{Ma21} 
is obtained.

\subsection{Proof of Theorem \ref{thm_2_1}} 

Theorem \ref{thm_2_1} is proved by almost the same argument as 
in the proof of \cite[Theorem 2.3]{Ma21} 
under the assumptions (M1), $\Omega \subset \R$, and the ${\rm LIC}(\Omega)$. 
Therefore, we only describe the outline of the proof. 
\bigskip

Let $c_\omega=|a(\omega)|$ and $\beta_\omega=\arg a(\omega)$. 
Then, 
\[
f(t) := f_\Pi(t) = \sum_{\omega \in \Omega}c_\omega \, e^{-i(t\omega-\beta_\omega)}. 
\]
We first consider the finite truncation
\[
f_N(t) = \sum_{|\omega| \leq N}c_\omega \, e^{-i(t\omega-\beta_\omega)} 
\quad (N \in \Z_{>0}). 
\]
Let $\T$ be the unit circle on $\C$, and $\T_N=\prod_{|\omega| \leq N} \T$. 
Define
\begin{equation} \label{eq_0903_3}
S_N(\mathbf{t}_N) = \sum_{|\omega| \leq N}c_\omega \, t_\omega ,
\end{equation}
where $\mathbf{t}_N =(t_\omega)_{|\omega| \leq N} \in \T_N$. 
Then obviously
\begin{equation} \label{f_206}
f_N(t) = S_N((e^{-i(t\omega-\beta_\omega)})_{|\omega| \leq N}). 
\end{equation}
These $f_N(t)$ and $S_N(\mathbf{t}_N) $ are analogs of (3.3) and (3.4) 
in \cite{Ma21}, respectively, but differ in that $\omega$ are not subscripted. 
For example, 
in \cite{Ma21}, $f_N(t)$ consists of $N$ terms, 
but $f_N(t)$ above consists of all terms of $\omega$ that satisfy $|\omega| \leq N$. 

\begin{proposition} \label{prop_2_1} 
We may construct a function $M_N: \C \to \R_{\geq 0}$, 
for which 
\[
\int_\C M_N(w) \Phi(w) |dw| = \int_{\T_N} \Phi(S_N(\mathbf{t}_N)) d^\ast \mathbf{t}_N
\]
holds for any continuous function $\Phi$ on $\C$, 
where $|dw|=dudv/(2\pi)$ for $w=u+iv$ 
and  $d^\ast \mathbf{t}_N$ 
is the normalized Haar measure on $\T_N$, 
that is the product measure of $d^\ast t=d\theta/(2\pi)$ 
for $t=e^{i\theta} \in \T$.
In particular, we obtain
\[
\int_\C M_N(w) |dw| =1. 
\]
Also, if two or more $\omega$'s satisfy $|\omega| \leq N$, 
the function $M_N(w)$ is compactly supported, 
nonnegative, and $M_N(\bar{w})=M_N(w)$. 
Moreover, if five or more $\omega$'s satisfy $|\omega| \leq N$, 
$M_N(w)$ is continuous. 
\end{proposition}

This is shown by the same argument as 
in the proof of \cite[Proposition 3.1]{Ma21}. 
The condition that 
there are two or more (resp. five or more) 
$\omega$'s that satisfy $|\omega| \leq N$ 
corresponds to the condition $N \geq 2$ (resp. $N \geq 5$) in \cite[Proposition 3.1]{Ma21} 
due to the difference in the definitions of $f_N(t)$ and $S_N(\mathbf{t}_N)$. 
The function $M_N(w)$ is constructed as a multiple convolution of 
\[
m_\omega(w) = \frac{1}{r}\delta(r-c_\omega), 
\quad w=re^{i\theta} \in \C, ~r=|w|, ~\theta=\arg w
\]
for $|\omega| \leq N$, 
where $\delta(\cdot)$ stands for the Dirac delta distribution. 
In particular, 
\begin{equation} \label{f_207}
{\rm supp}\,M_N \subseteq \left\{w \in \C ~:~|w| \leq \sum_{|\omega| \leq N} c_\omega\right\}
\end{equation}
if five or more $\omega$'s satisfy $|\omega| \leq N$. 
The inclusion follows from the fact that 
the support of $M_\Pi$ coincides with the image of $S_N(\mathbf{t}_N)$ 
(cf. \cite[Proposition 3.4]{Ma21}). 
From this fact, we find that \eqref{f_207} holds in general 
by considering definition~\eqref{eq_0903_3} geometrically. 
Moreover, the equality holds in \eqref{f_207} if 
\begin{equation} \label{eq_0903_4}
\max_{|\omega| \leq N} c_\omega
< \frac{1}{2} 
\sum_{|\omega| \leq N} c_\omega. 
\end{equation}
This can be shown as follows. 
(This type of argument originates from Bohr; 
see \cite[Hilfssatz 1]{Bo12}.)  
Assume that $\max_{|\omega| \leq N} c_{\omega}$ 
is attained by $\omega=\omega_1$, 
and $N$ is large enough to be $|\omega_1| \leq N$. 
When $t_{\omega_1}$ runs over $\mathbb{T}$,
the set of all $c_{\omega_1} t_{\omega_1}$ forms a circle $C_1$ whose
center is the origin and the radius is $c_{\omega_1}$.     For any other
$\omega_2$ with $|\omega_2| \leq N$, the set
\[   \{ c_{\omega_1}t_{\omega_1}+c_{\omega_2}t_{\omega_2}
             \mid  t_{\omega_1}, t_{\omega_2} \in \mathbb{T} \}
\]
is an annulus with 
inner radius $c_{\omega_1}-c_{\omega_2}$ and 
outer radius $c_{\omega_1}+c_{\omega_2}$. 
By adding a third $\omega_3$, 
we obtain a larger annulus, with 
inner radius $c_{\omega_1}-c_{\omega_2}-c_{\omega_3}$ and 
outer radius $c_{\omega_1}+c_{\omega_2}+c_{\omega_3}$.  
Repeating this procedure, 
we find that the interior of the circle $C_1$ is completely included in the image
of $S_N$ if condition \eqref{eq_0903_4} is satisfied and $N$ is sufficiently large, 
and hence the equality holds in \eqref{f_207}. 

By assumptions $\Omega \subset \R$ and the ${\rm LIC}(\Omega)$, 
the following holds by the same argument as in the proof of \cite[Proposition 4.2]{Ma21}. 

\begin{proposition} \label{prop_2_2} We have 
\[
\lim_{T\to\infty}\frac{1}{T}\int_{0}^{T}\Psi((e^{-i(t\omega-\beta_\omega)})_{|\omega| \leq N}) \, dt
= \int_{\T_N} \Psi(\mathbf{t}_N) d^\ast \mathbf{t}_N
\]
for any continuous $\Psi: \T_N \to \C$. 
\end{proposition}

In view of \eqref{f_206}, we have
\[
\lim_{T\to\infty}\frac{1}{T}\int_{0}^{T}\Phi(f_N(t)) \, dt
= \int_{\T_N} \Phi(S_N(\mathbf{t}_N)) d^\ast \mathbf{t}_N
\]
for any continuous function $\Phi$ on $\C$ by Proposition \ref{prop_2_2}. 
Then, combining this with Proposition \ref{prop_2_1}, we have 
\[
\lim_{T\to\infty}\frac{1}{T}\int_{0}^{T}\Phi(f_N(t)) \, dt
= \int_{\C} M_N(w) \Phi(w) |dw|
\]
for any continuous $\Phi$, which is the ``finite-truncation'' 
analog of Theorem \ref{thm_2_1}. 

Let $\psi_z(w)=\exp(i\Re(\bar{z}w))$, and define the Fourier transform of $m_\omega$ as 
\[
\widetilde{m}_\omega(z) := \int_\C m_\omega(w) \psi_z(w) |dw|. 
\]
We have
\begin{equation} \label{f_208}
\widetilde{m}_\omega(z) = J_0(c_\omega|z|)
\end{equation}
by \cite[Remark 5.1]{Ma21}. As stated there, this is originally due to Jessen and Wintner 
\cite[Section 5]{JW35} and Ihara \cite[Section 3.1]{I08}. 
If we define
\begin{equation} \label{f_209}
\widetilde{M}_N(z) := \prod_{|\omega| \leq N} \widetilde{m}_\omega(z), 
\end{equation}
then $\widetilde{M}_N(z)$ converges to $\widetilde{M}(z)$ 
uniformly in $\C$ as $N \to \infty$. 
The limit function $\widetilde{M_\Pi}(z):=\widetilde{M}(z)$ is continuous 
and belongs to $L^p$ (for any $p \in [1,\infty]$), 
and the above convergence is also $L^p$-convergence. 
The proof of these facts is the same with the proof of \cite[Proposition 5.5]{Ma21}, 
because it is applied by (M1) and the definition of $f_N(t)$. 
Now define 
\[
M_\Pi(w) := \int_\C \widetilde{M_\Pi}(z) \psi_{-w}(z) |dz|. 
\]
Then we obtain the following by the same argument as in the proof of \cite[Proposition 5.6]{Ma21} 
and \eqref{f_207}. 

\begin{proposition} \label{prop_2_3} 
When $N\to\infty$, $M_N(w)$ converges to $M_\Pi(w)$ uniformly in $w \in \C$. 
The limit function $M_\Pi(w)$ is continuous, nonnegative, compactly supported, 
$M_\Pi(\bar{w}) = M_\Pi(w)$, and 
\[
\int_\C M_\Pi(w) |dw| =1. 
\] 
The functions $M_\Pi$ and $\widetilde{M_\Pi}$ are Fourier duals of each other. In addition, 
\[
{\rm supp}\,M_\Pi \subseteq \left\{w \in \C~:~|w| \leq \sum_{\omega \in \Omega} c_\omega \right\}
\] 
and the equality holds if $\Pi$ satisfies \eqref{eq_0903_4} for $N=\infty$. 
\end{proposition}

Finally, by applying the argument in \cite[Section 6]{Ma21} 
to the above case, the proof of Theorem \ref{thm_2_1} is completed. 
\hfill $\Box$

\subsection{Examples satisfying condition \eqref{eq_0903_2}} \label{section_2_3} 
There are five nontrivial zeros of $\zeta(s)$ 
that satisfy $0<\Im(\rho)<35$ from a well-known numerical table. 
We order these zeros so that their imaginary parts increase, and denote them by 
$\rho_n=1/2-i\omega_n$ ($\omega_n \in \R_{<0}$, $1 \leq n \leq 5$). 
Then we verify that $\Pi=(\Omega_\zeta^+, a_H)$ 
satisfies \eqref{eq_0903_2} by 
\[
\frac{1}{2}\sum_{\Im(\rho)>0}\frac{m_\rho}{|\rho(\rho+1)|} 
>
\frac{1}{2}\sum_{n=1}^{5} \frac{1}{|\rho_n(\rho_n+1)|} = 0.00541 \ldots 
>
\frac{1}{|\rho_1(\rho_1+1)|} =  0.00497 \ldots. 
\]
Furthermore, we can verify that \eqref{eq_0903_2} also holds 
for $\Pi=(\Omega_\zeta^+, a_{H_\ell})$ for general $\ell \in \R$ as follows. 
It suffices to prove that the inequality
\[
\frac{1}{2}\sum_{n=1}^{5}\frac{1}{|(\rho_n-\ell)((1-\rho_n)-\ell)|} 
>
\frac{1}{|(\rho_1-\ell)((1-\rho_1)-\ell)|}
\]
holds for general $\ell$, that is equivalent to
\begin{equation} \label{eq_0904_1}
\sum_{n=2}^{5}\frac{1}{|(\rho_n-\ell)((1-\rho_n)-\ell)|} 
>
\frac{1}{|(\rho_1-\ell)((1-\rho_1)-\ell)|}.
\end{equation}
If $\rho=1/2-i\omega$ ($\omega \in \R$), we have
\[
\left|
\frac{1}{(\rho-\ell)((1-\rho)-\ell)} 
\right|
= \frac{1}{\omega^2+(\ell-1/2)^2}.
\]
Hence, inequality~\eqref{eq_0904_1} becomes
\[
\sum_{n=2}^5\frac{1}{\omega_n^2+(\ell-1/2)^2}
>
\frac{1}{\omega_1^2+(\ell-1/2)^2}.
\]
Setting $x=\ell-1/2$, define
\begin{equation} \label{eq_0904_2}
f(x) = \sum_{n=2}^5\frac{1}{\omega_n^2+x^2}
- \frac{1}{\omega_1^2+x^2}.
\end{equation}
It remains to show that $f(x)>0$ for all $x \geq 0$.  
For $x=0$ ($\ell=1/2$), we have
\[
0.00586 \ldots =
\sum_{n=2}^5 \frac{1}{\omega_n^2}
>
\frac{1}{\omega_1^2} = 0.00500 \ldots,
\]
hence $f(0)>0$. Extracting $x^2$ from the denominators in~\eqref{eq_0904_2}, we obtain
\[
f(x) = \frac{1}{x^2}
\left(
\sum_{n=2}^5 \frac{1}{1+(\omega_n/x)^2}
-
\frac{1}{1+(\omega_1/x)^2}
\right) 
\sim \frac{1}{x^2}\left(
\sum_{n=2}^5 1
-1
\right) = \frac{3}{x^2} > 0
\]
as $x \to \infty$. More precisely, using
\[
1 - x < \frac{1}{1+x} < 1 - x + x^2,
\]
we estimate $f(x)$ from below as
\[
\aligned 
\sum_{n=2}^5 \frac{1}{1+(\omega_n/x)^2}
-
\frac{1}{1+(\omega_1/x)^2}
& > \sum_{n=2}^5 \bigl(1-(\omega_n/x)^2\bigr)-1+(\omega_1/x)^2-(\omega_1/x)^4 \\[6pt]
& = 3 - \frac{1}{x^2}\left( \sum_{n=2}^{5}\omega_n^2 - \omega_1^2 \right) - \frac{\omega_1^4}{x^4}. 
\endaligned 
\]
For $x \geq 60$ we then have
\[
0<\frac{1}{x^2}\left( \sum_{n=2}^{5}\omega_n^2 + \omega_1^2 \right) + \frac{\omega_1^4}{x^4}<1,
\]
and hence
\[
f(x) > \frac{2}{x^2} >0 \qquad (x \geq 60).
\]
Finally, for $0 \leq x \leq 60$, one can directly verify numerically 
that $f(x)>0$. Therefore, $f(x)>0$ holds for all $x \geq 0$.

\section{A condition for $\Omega$ to be a subset of $\R$} \label{section_3}

For a pair $\Pi=(\Omega, a)$ satisfying (M1) and (M2), we define 
\begin{equation} \label{f_301}
g_{\Pi}(t):= f_\Pi(t)-f_\Pi(0)=\sum_{\omega \in \Omega} a(\omega) (e^{-it\omega}-1)
\end{equation}
for real $t$. 
The function $g_\Pi(t)$ is a continuous function on the real line 
with $g_\Pi(0)=0$, 
because the series on the right-hand side of  \eqref{f_301} 
converges absolutely and uniformly on any finite closed interval of $\R$ 
by (M1) and (M2). 

\begin{proposition} \label{prop_3_1} 
Let $\Pi=(\Omega,a)$ be a pair satisfying (M1) and (M2). 
Then, the real part $\Re(g_\Pi(t))$ is bounded on $[0,\infty)$ if and only if $\Omega$ is a subset of 
the closed lower half-plane $\C\setminus\C_+=\{z  \,|\,\Im(z) \leq 0\}$. 
\end{proposition}
\begin{proof} 
If $\Omega \subset \C\setminus\C_+$, 
the function $g_\Pi(t)$ is obviously bounded on $[0,\infty)$ 
by definition \eqref{f_301}. 
Therefore, $\Re(g_\Pi(t))$ is also bounded on $[0,\infty)$. 
To show the converse, 
we consider the Fourier transform of $\Re(g_\Pi(t))\mathbf{1}_{\geq 0}(t)$. 
By the integral formula
\[
\int_{0}^{\infty}  (e^{-it\omega}-1)\,e^{izt}\,dt
= \frac{i\omega}{z(z-\omega)} \quad \text{for} \quad  \Im(z)>\max\{0, \Im(\omega)\}, 
\]
we obtain
\begin{equation} \label{f_302}
\int_{0}^{\infty}  \Re(g_\Pi(t))\,e^{izt}\,dt
= \int_{0}^{\infty}  \frac{1}{2}(g_\Pi(t)+\overline{g_\Pi(t)})\,e^{izt}\,dt
= -\frac{i}{z^2} Q_\Pi(z)
\end{equation}
for $\Im(z) > c$, where 
\[
Q_\Pi(z) := \frac{1}{2}
\left[
 \sum_{\omega \in \Omega} a(\omega) \frac{-z\omega}{z-\omega} 
+ \sum_{\omega \in \Omega} \overline{a(\omega)} \frac{-z(-\bar{\omega})}{z-(-\bar{\omega})}  
\right]
\]
and $c \geq 0$ is the constant in (M2). 
By definition and (M2), $Q_\Pi(z)$ is meromorphic on $\C$ 
and holomorphic on $\C \setminus (\Omega \cup (-\overline{\Omega}))$. 

Suppose that $\Re(g_\Pi(t))$ is bounded on $[0,\infty)$. 
Then, the left-hand side of \eqref{f_302} 
converges absolutely and uniformly on any compact subset 
in the upper half-plane $\C_+= \{z\,|\,\Im(z) > 0\}$. 
Therefore, $Q_\Pi(z)$ is holomorphic in $\C_+$. 
However, if $\Omega$ has an element in $\C_+$, 
it must be a pole of $Q_\Pi(z)$, 
since $a(\omega)\not=0$ for every $\omega \in \Omega$ by definition.   
This is a contradiction. 
Hence, $\Omega \subset \C\setminus\C_+$ if $\Re(g_\Pi(t))$ is bounded on $[0,\infty)$.
\end{proof}

To prove  an extension of Theorem \ref{thm_1_2} under the general setting in Section \ref{sec_2_1},  
we add the following third condition. 
\begin{enumerate}
\item[(M3)] The set $\Omega$ is closed under the complex conjugation 
$\omega \mapsto \bar{\omega}$. 
\end{enumerate}
Condition (M3) is satisfied by both $(\Omega_\zeta^+,\,a_H)$ and $(\Omega_\zeta^+,\,a_{H_1})$, 
because, if $\rho=1/2-i\omega$ is a nontrivial zero of $\zeta(s)$, 
$1-\bar{\rho}=1/2-i\bar{\omega}$ is also a zero with the same multiplicity. 

Due to the symmetry of $\Omega$ in (M3), the following result immediately follows from Proposition  \ref{prop_3_1}.

\begin{corollary} \label{cor_3_1}
Let $\Pi=(\Omega,a)$ be a pair satisfying (M1), (M2), and (M3). 
Then, the real part $\Re(g_\Pi(t))$ is bounded on $[0,\infty)$ 
if and only if $\Omega$ is a subset of $\R$. 
\end{corollary}

\begin{remark}
In Corollary \ref{cor_3_1}, 
if $\Re(g_\Pi(t))$ is bounded on $[0,\infty)$, 
then $\Omega \subset \R$, 
which implies that $\Re(g_\Pi(t))$ is bounded on $\R$ 
by \eqref{f_301}.  
Therefore, $[0,\infty)$ in Corollary \ref{cor_3_1} can be replaced by $\R$, 
but the behavior of $\Re(g_\Pi(t))$ on $(-\infty,0)$ is not necessary below. 
\end{remark}

Using Corollary \ref{cor_3_1}, we obtain the following result.

\begin{theorem} \label{thm_3_1}
Let $\Pi=(\Omega,a)$ be a pair satisfying 
(M1), (M2), and (M3). 
We assume that equality \eqref{f_205} 
holds for any test function $\Phi: \R \to \C$ which is locally Riemann integrable. 
Then $\Omega$ is a subset of $\R$. 
\end{theorem}
\begin{proof}
By Corollary \ref{cor_3_1}, 
it is sufficient to prove that 
$\Re(f_\Pi(t))$ is bounded on $[0,\infty)$ if \eqref{f_205} 
holds for any locally Riemann integrable function $\Phi$. 
We prove the boundedness of $\Re(f_\Pi(t))$ on $[0,\infty)$ by contradiction. 
We denote $|\Re(f_\Pi(t))|$ by $F(t)$ for simplicity. 
First, we note that there exists $R>0$ such that 
\[
{\rm supp}\,M_\Pi^\Re \subset \{u \in \R\,:\, |u| \leq R\}, 
\]
since the support of $M_\Pi^\Re(u)$ is compact by Corollary \ref{cor_2_1}. 
We assume that $F(t)\,(\geq 0)$ is unbounded on $[0,\infty)$. 
Then, there exists $t_1>0$ that satisfies 
\[
F(t_1) > R+1 \quad \text{and} \quad  t_1 > R+1. 
\]
For this, we write $a_1=F(t_1)$. 
Subsequently, again from the unboundedness, 
there exists $t_2 > t_1+2$ that satisfies $F(t_2) > a_1 +2$. 
For this, we write $a_2=F(t_2)$. 
By repeating the above process, 
we obtain sequences $\{t_n\}_{n \geq 1}$ and $\{a_n\}_{n \geq 1}$ 
that satisfy $a_n=F(t_n)$, 
\[
t_{n+1}> t_n+2, \quad \text{and} \quad a_{n+1}=F(t_{n+1}) > a_n+2. 
\]
By definition, these two sequences are strictly monotonically increasing. 

For each $t_n$, there exists $0<\delta_n<1$ such that 
$|t-t_n| <\delta_n$ implies $|F(t)-a_n| <1$
by the continuity of $F(t)$. 
Note that no $t \in \R$ satisfies $|t-t_n|<1$ for two different $n$ simultaneously, 
because $t_{n+1}>  t_n+2$.  

Now we define the test function $\widetilde{\Phi}$ by 
\[
\widetilde{\Phi}(u) :=
\begin{cases}
~ \displaystyle{\frac{t_n}{\delta_n}} & \text{if $|u-a_n|<1$}, \\[8pt]
~ 0 & \text{otherwise}. 
\end{cases}
\]
This $\widetilde{\Phi}$ is clearly a Riemann integrable function, with the support
\[
\bigcup_{n=1}^{\infty} [a_n-1,a_n+1]. 
\] 
Note that no $u \in \R$ satisfies $|u-a_n|<1$ for two different $n$ simultaneously, 
because $a_{n+1}>a_n+2$ by definition.  

For the test function $\Phi(u):=\widetilde{\Phi}(|u|)$, 
the value of the right-hand side of \eqref{f_205} is zero, 
since the support of $\Phi$ is contained in $\{u >R\}$ 
by $a_1-1=F(t_1)-1>R+1-1$, and $M_\Pi^\Re(u)=0$ if $u >R$. 
However, the value of the left-hand side of \eqref{f_205} is not zero as follows.

By definition of $\delta_n$, if $|t-t_n|<\delta_n$, 
then 
$|F(t)-a_n|<1$, and therefore $\widetilde{\Phi}(F(t))=t_n/\delta_n$. 
Hence, for $T=t_N+\delta_N$, 
\[
\aligned 
\frac{1}{T}\int_{0}^{T} \Phi(\Re(f_\Pi(t))) \, dt
& = 
\frac{1}{t_N+\delta_N} \int_{0}^{t_N+\delta_N} \widetilde{\Phi}(F(t)) \, dt \\
& \geq 
\frac{1}{t_N+\delta_N} \sum_{n=1}^{N} \int_{t_n-\delta_n}^{t_n+\delta_n} \widetilde{\Phi}(F(t)) \, dt \\
& =
\frac{1}{t_N+\delta_N} \sum_{n=1}^{N} \int_{t_n-\delta_n}^{t_n+\delta_n} \frac{t_n}{\delta_n} \, dt 
 = 
\frac{1}{t_N+\delta_N}  \sum_{n=1}^{N}\frac{t_n}{\delta_n} \cdot 2\delta_n \\
& > 
\frac{2t_N}{t_N+\delta_N} >1.
\endaligned 
\] 
The final inequality on the right-hand side follows 
from $\delta_N<1$ and $t_N \geq t_1  > R+1>1$.
That is, 
$T^{-1}\int_{0}^{T} \Phi(F(t)) \, dt >1$ 
for the increasing sequence $T=t_N+\delta_N$, 
and hence, 
the left-hand side of \eqref{f_205} cannot have zero as the limit. 
This is a contradiction. 
\end{proof}

Applying Theorem \ref{thm_3_1} to $(\Omega_\zeta^+, a_{H})$, 
we obtain the following, because $\Omega_\zeta^+ \subset \R$ implies the Riemann hypothesis. 

\begin{corollary} \label{cor_3_2}
Theorem \ref{thm_1_2} holds. 
\end{corollary}

The (obvious) analog of Theorem \ref{thm_1_2} for $H_1(X)$ 
is obtained by applying Theorem \ref{thm_3_1} to $(\Omega_\zeta^+, a_{H_1})$. 
(cf. The comment after Corollary \ref{cor_2_1} at the end of Section \ref{sec_2_1}.) 
Furthermore, 
the following result is also proven by replacing 
$\Re(f_\Pi(t))$, $M_\Pi^\Re(u)$, and 
$\widetilde{\Phi}(|u|)$ ($u \in \R$) 
with 
$f_\Pi(t)$, $M_\Pi(w)$, and 
$\widetilde{\Phi}(|w|)$ ($w \in \C$), respectively, 
in the proof of Theorem \ref{thm_3_1}.

\begin{theorem} 
Let $\Pi=(\Omega,a)$ be a pair satisfying 
(M1), (M2), and (M3).  
We assume that equality \eqref{f_203} 
holds for any test function $\Phi: \C \to \C$ which is locally Riemann integrable. 
Then $f_\Pi(t)$ is bounded on $[0,\infty)$. 
\end{theorem}

\section{A class of screw functions} \label{section_4}

To relate the functions \eqref{f_202} attached to pairs $\Pi=(\Omega,a)$ 
with screw functions, we consider the following conditions: 
\begin{enumerate}
\item[(S1)] 
$a(\bar{\omega})=\overline{a(\omega)}$ for any $\omega \in \Omega$, 
\item[(S2)] $\Omega \subset \R$ and 
$a(\omega) >0$ for all $\omega \in \Omega$. 
\end{enumerate}
The former is a weaker condition than the latter, since (S2) implies (S1). 
The pairs $(\Omega_\zeta^+,\,a_{H_\ell})$ satisfy  (S1) unconditionally 
and (S2) under the Riemann hypothesis. 
The pair 
$(\Omega_\zeta^+,\,a_H)$ satisfies neither (S1) nor (S2) 
even if assuming the Riemann hypothesis. 
We have already noted that (M1) and (M2) imply 
the continuity of $g_\Pi(t)$ and $g_\Pi(0)=0$, 
but if we assume (M3) and (S1) in addition, $g_\Pi(t)$ satisfies 
\begin{equation} \label{f_401}
g_\Pi(-t)=\overline{g_\Pi(t)}.
\end{equation}

\begin{proposition} \label{prop_4_1} 
Let $\Pi=(\Omega,a)$ be a pair satisfying  (M1), (M2), (M3), and (S1).  
Then, $g_\Pi(t)$ defined in \eqref{f_301}  is a screw function on $\R$ 
if and only if $\Pi$ satisfies (S2).  
\end{proposition}
\begin{proof}
First,  we prove that $g_\Pi(t)$ is a screw function if $\Pi$ satisfies (S2). 
It suffices to show that $G_g(t,u)$ is nonnegative definite on $\R$, 
since we already confirm \eqref{f_401}. 
We have 
\[
G_{g}(t,u) = \sum_{\omega \in \Omega} a(\omega) (e^{-it\omega}-1)(e^{iu\omega}-1)
\]
by a direct calculation, and therefore
\begin{equation} \label{f_402}
\aligned 
\sum_{i,j=1}^{n} G_g(t_i,t_j) \,  \xi_i \overline{\xi_j} 
& = 
\sum_{\omega \in \Omega} a(\omega) 
\left|\sum_{i=1}^{n}  (e^{-it_i\omega}-1)\xi_i \right|^2 \geq 0
\endaligned 
\end{equation}
for any  $n \in \Z_{>0}$,  $t_i \in \R$ and $\xi_i \in \C$ by (S2). 
Hence, $G_g(t,u)$ is nonnegative definite on $\R$. 

Conversely, we suppose that $g_\Pi(t)$ is a screw function on $\R$. 
We define the function $Q(z)$ by using the right Fourier integral as 
\begin{equation*}
\int_{0}^{\infty}  g_\Pi(t)\,e^{izt}\,dt
= -\frac{i}{z^2} Q(z).
\end{equation*}
By (M1), (M2), and \eqref{f_301}, 
the left-hand side is integrated term by term, 
defining a holomorphic function on $\Im(z)>c$ such that
\begin{equation*}
Q(z) = 
 \sum_{\omega \in \Omega} a(\omega) \frac{-z\omega}{z-\omega}
\end{equation*} 
holds. 
Furthermore, the assumption implies that 
$Q(z)$ extends to a holomorphic function defined on $\C_+$ 
mapping $\C_+$ into $\C_+ \cup \R$ by \cite[Satz 5.9]{KrLa77}. 
Hence, $\Omega \subset \R$ is shown in the same way 
as in the proof of Proposition \ref{prop_3_1} and Corollary \ref{cor_3_1}. 
By the definition of screw functions, 
$G_g(t,u)$ must be nonnegative definite, 
which implies that $a(\omega)>0$ for all $\omega \in \Omega$ by \eqref{f_402}. 
\end{proof}

Applying Proposition \ref{prop_4_1} to 
$\Omega=\Omega_\zeta^+ \cup (-\Omega_\zeta^+)$ 
and $a=\tfrac{1}{2}a_{H_\ell}$, 
we obtain the following, 
where the right-hand side of $a$ is an obvious extension to $\Omega$.

\begin{corollary} \label{cor_4_1} 
The function 
\[
g_{H_\ell}(t) := H_\ell(e^t) - H_\ell(1)
\]
is a screw function on $\R$ 
if and only if the Riemann hypothesis holds. 
In particular, Theorem \ref{thm_1_3} holds. 
\end{corollary}

\section{A point mass formula at the origin for \\ infinitely divisible distributions} \label{section_5}

In this section we discuss the connection between the $M$-functions and 
the theory of infinitely divisible distributions, with the aid of screw functions.
First, we review the following result:

\begin{proposition} \label{prop_501} 
For a function $h(t)$ on $\R$, 
$\exp(h(t))$ is the characteristic function 
of an infinitely divisible distribution $\mu$ on $\R$: 
$\exp(h(t)) =  \int_{-\infty}^{\infty} e^{itx} \mu(dx)$,
if and only if 
\begin{equation} \label{f_501}
h(t) = -\frac{1}{2}At^2 + i B t + 
\int_{-\infty}^{\infty} 
\left(
e^{it\omega} -1 -\frac{it\omega}{1+\omega^2} 
\right)d\nu(\omega)
\end{equation}
for some $A \geq 0$, $B \in \R$, and a measure $\nu$ on $\R$ satisfying 
\[
\nu(\{0\})=0 \quad \text{and} \quad 
\int_{-\infty}^{\infty}{\rm min}(1,\omega^2)d\nu(\omega)<\infty. 
\]
If the measure $\nu$ in \eqref{f_501} satisfies 
$\int_{|\omega| \leq 1}|\omega|d\nu(\omega)< \infty$, 
then \eqref{f_501} can be rewritten as 
\begin{equation} \label{f_502}
h(t) = -\frac{1}{2}At^2 + i B_0 t + 
\int_{-\infty}^{\infty} (e^{it\omega} -1)d\nu(\omega)
\end{equation}
for some $B_0 \in \R$. 
\end{proposition}
\begin{proof}
The first half is obtained by applying \cite[Theorem 8.1 and Remark 8.4]{Sa99} to the case of $\R$, 
and the second half is obtained by applying \cite[(8.7)]{Sa99} to the case of $\R$. 
\end{proof}

Using Proposition \ref{prop_501}, we obtain: 

\begin{proposition} \label{prop_502} 
Let $\Pi=(\Omega,a)$ be a pair satisfying (M1) and (S2).  
Then, for any $y>0$, there exists an infinitely divisible distribution $\mu_{\Pi,y}(x)$ on $\R$ 
whose characteristic function is $\exp(y\Re(g_{\Pi}(t)))$, that is, 
\begin{equation} \label{f_503}
\exp(y \Re(g_\Pi(t))) = \int_\R e^{itx} \mu_{\Pi,y}(dx).
\end{equation}
\end{proposition}
\begin{proof}
First, we note that $\Pi=(\Omega,a)$ satisfies (M1), (M2), (M3), (S1), and (S2) 
because (M2), (M3), and (S1) follow from (S2). 
For any $y>0$, $y\Re(g_\Pi(t))$ is a real-valued screw function 
satisfying $y\Re(g_\Pi(0))=0$ and $f_\Pi(0)>0$ 
by assumptions, \eqref{f_301}, and Proposition \ref{prop_4_1}. 
Therefore, $y\Re(g_\Pi(t))$ has the form \eqref{f_501} 
with $A=0$ by \cite[Theorem 5.1]{KrLa14}. 
Then, for any $y>0$, 
there exists an infinitely divisible distribution $\mu_{\Pi,y}(x)$ 
such that \eqref{f_503} holds by Proposition \ref{prop_501}. 
\end{proof}

\begin{remark} 
In the proof of Proposition \ref{prop_502}, 
we referred to \cite[Theorem 5.1]{KrLa14} to prove \eqref{f_503}, 
but if we use \eqref{f_502} based on \eqref{f_301}, 
the result in \cite{KrLa14} is not necessary. 
Such an argument is similar to that made in \cite[Proof of (1)$\Rightarrow$(2) in Theorem 1.1]{NaSu23}.
However, in order to clarify the relation 
between screw functions and infinitely divisible distributions for the readers, 
we provided a proof using \cite{KrLa14}.
\end{remark}

Theorem \ref{thm_1_4} is obtained by applying the following result 
to $\Pi=(\Omega_\zeta^+, a_{H_1})$ 
under the Riemann hypothesis and ${\rm LIC}(\Omega_\zeta^+)$ 
since $H_1(e^t)=\Re(f_\Pi(t))$. 

\begin{theorem} \label{thm_5_1} 
Let $\Pi=(\Omega,a)$ be a pair satisfying (M1) and (S2). 
We assume ${\rm LIC}(\Omega)$. 
For $y>0$, let $\mu_{\Pi,y}(x)$ be 
the infinitely divisible distribution in Proposition \ref{prop_502}, 
and 
let $M_{\Pi}^\Re(u)$ be the $M$-function in Corollary \ref{cor_2_1}.
Then the value of the point mass of $\mu_{\Pi,y}(x)$ at the origin 
is given by the $M$-function as follows: 
\begin{equation} \label{f_504}
\aligned 
\mu_{\Pi,y}(\{0\}) 
= \exp(-y\, f_\Pi(0))\widetilde{M_\Pi^\Re}(-iy)
= \exp(-y\, f_\Pi(0))\prod_{\omega \in \Omega} J_0(iy|a(\omega)|),
\endaligned 
\end{equation}
where 
\[
\widetilde{M_\Pi^\Re}(z)
= \frac{1}{\sqrt{2\pi}}\int_\R M_\Pi^\Re(u)e^{izu} \, du \quad (z \in \C). 
\] 
\end{theorem}
\begin{proof} 
The pair $\Pi=(\Omega,a)$ satisfies (M1), (M2), (M3), (S1), and (S2) 
as in the proof of Proposition \ref{prop_502}. 
We have 
\[
\widetilde{M_\Pi}(z) = \int_\C M_\Pi(w) \exp(i\Re(\bar{z}w)) \, |dw| 
= \prod_{\omega \in \Omega} J_0(|z||a(\omega)|) \quad (z \in \C)
\]
from \eqref{f_208}, \eqref{f_209}, and 
Proposition \ref{prop_2_3}. 
Restricting this equation to real $z$, we get
\begin{equation} \label{f_505}
\frac{1}{\sqrt{2\pi}}\int_\R M_\Pi^\Re(u) e^{izu} \, du
= \prod_{\omega \in \Omega} J_0(z|a(\omega)|) \quad (z \in \R)
\end{equation}
by \eqref{f_204}, 
because the power series expansion of $J_0(x)$ at the origin consists of even powers of $x$. 
The left-hand side of \eqref{f_505} is the definition of $\widetilde{M_\Pi^\Re}(z)$ for real $z$ 
and extends to $z \in \C$ by the compactness of the support of $M_\Pi^\Re(u)$ in Corollary \ref{cor_2_1}. 
The right-hand side of \eqref{f_505} also extends to $z \in \C$,   
because $J_0(x)
=1 + O(|x|)$ 
as $|x| \to 0$. Hence, 
\begin{equation} \label{f_506}
\widetilde{M_\Pi^\Re}(z)
= \frac{1}{\sqrt{2\pi}}\int_\R M_\Pi^\Re(u) \exp(izu) \, du
= \prod_{\omega \in \Omega} J_0(z|a(\omega)|)
\end{equation}
holds for $z \in \C$. 

On the other hand, for $x \in \R$, we have 
\begin{equation} \label{f_507}
\aligned 
\mu_{\Pi,y}(\{x\}) 
& = \lim_{T \to \infty} \frac{1}{2T}\int_{-T}^{T} \exp(y\, \Re(g_\Pi(t)))  e^{-ixt} \, dt \\
& = \exp(-y\, f_\Pi(0))\lim_{T \to \infty} \frac{1}{T}\int_{0}^{T} \exp(y\, \Re(f_\Pi(t)))  e^{-ixt} \, dt
\endaligned 
\end{equation}
by \eqref{f_503} and  the inversion formula \cite[Theorem 3.10.4]{Du19}, 
since $\Re(g_\Pi(t))$ is even by \eqref{f_401} 
and equals to $\Re(f_\Pi(t))-f_\Pi(0)$ by \eqref{f_301}, (M3), and (S1). 
If we take $x=0$ in \eqref{f_507}, the right-hand side  is 
\[
\aligned 
\exp(-y\, f_\Pi(0)) & \frac{1}{\sqrt{2\pi}}\int_\R M_\Pi^\Re(u) \exp(i(-iy)u) \, du \\
& = \exp(-y\, f_\Pi(0)) \widetilde{M_\Pi^\Re}(-iy) =\exp(-y\, f_\Pi(0)) \prod_{\omega \in \Omega} J_0(iy|a(\omega)|)
\endaligned 
\]
by \eqref{f_205} for $\Phi(u)=\exp(yu)$ and \eqref{f_506}, since $J_0(x)$ is even. 
Therefore,  we obtain \eqref{f_504}.  
\end{proof}

\section{Explicit formulas for $H(X)$ and $H_\ell(X)$} \label{section_6}

The boundedness of the series $H(X)$ and $H_1(X)$ can be observed 
without using the information about the nontrivial zeros, 
at least numerically, by the following formulas (Proposition \ref{prop_6_1}). 
As seen in the proof, 
they are obtained by just combining classical results \cite[Section 17, (1)]{Da80}, 
\cite[(2.6)]{Gu48}, and \cite[p. 81]{In64}. 
Also, \eqref{f_601} below is essentially the same as Fujii's \eqref{f_103}, which is obtained by integrating formula \eqref{f_605}. Similarly, 
\eqref{f_602} below is essentially a special case of Ihara, Murty, and Shimura \cite[Theorem 1]{IMS09}.
In this sense, these two formulas are not new. 
However, the formula for $H_1(X)$ is extended to the more general cases of $H_\ell(X)$ 
(Propositions \ref{prop_6_2} and \ref{prop_6_3}). 
This provides an alternative proof of \cite[Theorem 1.1 (2)]{Su23}, 
which is used to establish several key results in \cite{Su23}.

\begin{proposition} \label{prop_6_1} 
The following formulas hold unconditionally for $X>1$: 
\begin{equation} \label{f_601}
\aligned 
H(X) 
& =  \frac{1}{2}  \sqrt{X} - \frac{1}{\sqrt{X}} \sum_{n \leq X} 
\Lambda(n)\left( 1 - \frac{n}{X} \right)  \\
& \quad 
- \frac{1}{\sqrt{X}} \log 2\pi - \frac{1}{X\sqrt{X}} \, 12 \zeta'(-1)
 \\
& \quad  
- \frac{1}{2\sqrt{X}} 
\left[
\log(1-X^{-2}) 
+ \frac{1}{X} \log \frac{X+1}{X-1} 
\right] ,
\endaligned 
\end{equation}

\begin{equation} \label{f_602}
\aligned 
H_1(X) 
&=  \sum_{n \leq X}\frac{\Lambda(n)}{\sqrt{n}} 
\left( \sqrt{\frac{X}{n}} - \sqrt{\frac{n}{X}} \right) - 
\sqrt{X} \,(\log X - C_0-1)  \\
& \quad -  \frac{1}{\sqrt{X}} \log 2\pi 
-  \frac{1}{\sqrt{X}} \left[ \frac{1}{2}\log(1-X^{-2}) + \frac{X}{2}\log\frac{X+1}{X-1} -1\right]. 
\endaligned 
\end{equation}
Furthermore, 
formulas \eqref{f_601} and \eqref{f_602} hold for $X=1$ 
in the sense that the right limits of the right-hand sides at $X=1$ 
are equal to the values of the left-hand sides at $X = 1$. 
\end{proposition}
\begin{proof} We obtain 

\begin{equation} \label{f_603}
H(X) 
= \frac{1}{\sqrt{X}} \sum_{\rho}\frac{X^{\rho}}{\rho} 
- \frac{1}{X\sqrt{X}} \sum_{\rho}\frac{X^{\rho+1}}{\rho+1}
\end{equation}
and
\begin{equation} \label{f_604}
H_1(X)= \frac{1}{\sqrt{X}} \sum_{\rho} \frac{X^{\rho}}{\rho} 
- \sqrt{X}\sum_{\rho} \frac{X^{\rho-1}}{\rho-1}
\end{equation}
by partial fraction decomposition, 
where the sum is understood as 
\[
\displaystyle{\sum_\rho = \lim_{T \to \infty}\sum_{|\Im(\rho)| \leq T}}
\] 
as usual. 
On the other hand, it is known that 
\begin{equation} \label{f_605}
\sum_{\rho}  \frac{X^{\rho}}{\rho}
= X 
- \sideset{}{'}\sum_{n \leq X} \Lambda(n)
- \log 2\pi - \frac{1}{2}\log(1-X^{-2})
\end{equation}
for $X>1$ by \cite[Section 17, (1)]{Da80}, 
\begin{equation} \label{f_606}
\sum_{\rho} \frac{X^{\rho+1}}{\rho+1} 
= \frac{1}{2}X^2
 - \sideset{}{'}\sum_{n \leq X} n\Lambda(n)
+12 \zeta'(-1)
+ \frac{1}{2} \log \frac{X+1}{X-1}
\end{equation}
for $X>1$ by \cite[(2.6)]{Gu48}, 
and 
\begin{equation} \label{f_607}
\sum_{\rho}  \frac{X^{\rho-1}}{\rho-1}
= \log X 
- \sideset{}{'}\sum_{n \leq X}\frac{\Lambda(n)}{n} 
- C_0
- \frac{1}{X} + \frac{1}{2}\log\frac{X+1}{X-1} 
\end{equation}
for $X>1$ by \cite[p. 81]{In64} with the replacements $1/x \mapsto X$ and $\rho \mapsto 1 -\rho$, 
where $\displaystyle{\sideset{}{'}\sum_{n \leq X}a_n}$ means 
$\displaystyle{\sum_{n \leq X}a_n-\frac{1}{2}a_X}$ when $X$ is a prime power 
and $\displaystyle{\sum_{n \leq X}a_n}$ otherwise.  
Applying \eqref{f_605} and \eqref{f_606} to \eqref{f_603}, 
we obtain \eqref{f_601} for $X>1$. 
Applying \eqref{f_605} and \eqref{f_607} to \eqref{f_604}, 
we obtain \eqref{f_602} for $X>1$. 

By definition \eqref{f_101} (resp. \eqref{f_109}), 
the left-hand side of \eqref{f_601} (resp. \eqref{f_602}) is right continuous at $X=1$. 
On the other hand, the right-hand side of \eqref{f_601} (resp. \eqref{f_602}) 
has the right limit at $X=1$, because 
\[
\log(1-X^{-2}) 
+ \frac{1}{X} \log \frac{X+1}{X-1} 
= 2 \log\frac{X+1}{X}+\frac{X-1}{X}\log \frac{X-1}{X+1}
\]
\[
\left(\text{resp.} \quad
\aligned  
\frac{1}{2}\log(1-X^{-2}) & + \frac{X}{2}\log\frac{X+1}{X-1} \\
& =  \frac{X+1}{2}\log(X+1)- \log X-\frac{X-1}{2}\log(X-1) 
\endaligned 
\right).
\]
Therefore, \eqref{f_601} and \eqref{f_602} hold in the sense stated in the proposition. 

Finally, we outline the derivation of \eqref{f_601} from \eqref{f_103}. 
The left-hand side of \eqref{f_103} can be written as 
$\sum_{n \leq X} \Lambda(n)(X-n) - X^2/2$, 
and the right-hand side becomes
\[
\aligned 
- X^{3/2}H(X)
-X \log 2\pi 
-12\zeta'(-1)- X \sum_{n=1}^{\infty} \frac{X^{-2n}}{2n(2n-1)}
\endaligned 
\]
using the well known formulas $\zeta'/\zeta(0)=\log 2\pi$ 
and $\zeta(-1)=-1/12$. 
Here, we note that 
\[
\aligned 
X & \sum_{n=1}^{\infty} \frac{X^{-2n}}{2n(2n-1)}
 = -\sum_{n=1}^{\infty} \frac{X^{-2n+1}}{2n} + \sum_{n=1}^{\infty} \frac{X^{-2n+1}}{2n-1} \\
& = \frac{X}{2} \log(1-X^{-2}) + {\rm arctanh}\, \frac{1}{X} 
= 
\frac{X}{2} 
\left[
\log(1-X^{-2}) 
+ \frac{1}{X} \log \frac{X+1}{X-1}
\right].
\endaligned 
\] 
Thus, \eqref{f_103} can be rewritten as \eqref{f_601}.
\end{proof}
We obtain 
\[
H(1) 
= \sum_\rho \frac{1}{\rho(\rho+1)} 
= \frac{1}{2} - \log 4 \pi -12\zeta'(-1) = -0.045970 \ldots
\]
by taking the limit $X \to 1+0$ on the right side of \eqref{f_601}. 
On the other hand, 
noting the symmetry of nontrivial zeros for $\rho \mapsto 1 -\rho$,
\[
H_1(1) = 2 \sum_{\rho} \frac{1}{\rho} = C_0+2 -\log 4\pi = 0.046191 \ldots
\]
by taking the limit $X \to 1+0$ on the right side of \eqref{f_602}. 
This is a well-known equation found in \cite[Section 12, (10) and (11)]{Da80}, for example.
\medskip

Formula \eqref{f_602} is generalized to $H_\ell(X)$ as follows. 

\begin{proposition} \label{prop_6_2} 
Let $\ell$ be a real number that is not equal to any of 
\[
-2n,  ~0, ~\frac{1}{2}, ~1, ~2n+1 \quad (n \in \Z_{>0}).
\] 
Then the following formula holds unconditionally for $X>1$: 
\begin{equation} \label{f_0315_1}
\aligned 
H_\ell(X)
& = \frac{X^{1/2}}{\ell(\ell-1)}  
-  \sum_{n \leq X}\frac{\Lambda(n)}{\sqrt{n}} 
\frac{1}{1-2\ell} \left[ \left( \frac{X}{n} \right)^{\ell-1/2} - \left( \frac{X}{n} \right)^{-(\ell-1/2)} \right] 
\\
& \quad 
- \frac{1}{1-2\ell} \left[ \frac{\zeta'}{\zeta}(\ell)X^{\ell-1/2} 
- \frac{\zeta'}{\zeta}(1-\ell)X^{-(\ell-1/2)} \right] \\
& \quad 
+ \frac{X^{-1/2}}{2}\cdot 
\frac{X^{-2}}{1-2\ell} \Bigl[ 
\Phi(X^{-2},1,1+\ell/2) \\ 
& \qquad \qquad \qquad \qquad \qquad \qquad - \Phi(X^{-2},1,1+(1-\ell)/2) \Bigr],
\endaligned 
\end{equation}
where $\Phi(z,s,a) = \sum_{n=0}^{\infty} z^n(n+a)^{-s}$ 
is the Hurwitz-Lerch zeta function for $|z|<1$ and $a \not=0, -1,-2,\cdots$. 
Furthermore, formula \eqref{f_0315_1} holds for $X=1$ 
in the sense that the right limit of the right-hand side at $X=1$ 
is equal to the value of the left-hand side at $X = 1$. 
\end{proposition}
\begin{proof}
We have
\begin{equation} \label{f_608}
\aligned 
H_\ell(X) 
& = \frac{X^{\ell-1/2}}{1-2\ell} 
\sum_\rho \frac{X^{\rho-\ell}}{\rho-\ell} 
- 
\frac{X^{1/2-\ell}}{1-2\ell} 
\sum_\rho \frac{X^{\rho-(1-\ell)}}{\rho-(1-\ell)} 
\endaligned 
\end{equation}
by partial fraction decomposition since $\ell \not=1/2$. 
On the other hand,  
\begin{equation} \label{f_609}
\sum_{\rho} \frac{X^{\rho-s}}{\rho-s} 
= 
 - \sideset{}{'}\sum_{n \leq X} \frac{\Lambda(n)}{n^s} 
+ \frac{X^{1-s}}{1-s} 
-\frac{\zeta'(s)}{\zeta(s)}
+ \frac{1}{2}X^{-s-2} \Phi(X^{-2},1,1+\frac{s}{2})
\end{equation}
for $X>1$, $s\not=1$, $\rho$, $-2n$ ($n \in \Z_{>0}$) by \cite[(2.6)]{Gu48} 
and \cite[1.11 (1)]{Er81}. 
Applying \eqref{f_609} to \eqref{f_608}, 
we get \eqref{f_0315_1} for $X>1$ since $\ell \not=-2n,  0, 1/2, 1, 2n+1$ ($n \in \Z_{>0}$). 

By definition \eqref{f_109}, 
the left-hand side of \eqref{f_0315_1}  is right continuous at $X=1$. 
On the other hand, the right-hand side of \eqref{f_0315_1} 
has the right limit at $X=1$, because 
\[
\aligned 
\Phi(X^{-2},1,1+\ell/2) &- \Phi(X^{-2},1,1+(1-\ell)/2) \\
&= -\left(\ell-\frac{1}{2}\right) \sum_{n=0}^{\infty} \frac{X^{-2n}}{(n+1+\ell/2)(n+1+(1-\ell)/2)}.
\endaligned 
\]
Hence we complete the proof. 
\end{proof}

\begin{proposition} \label{prop_6_3}
The following formula holds unconditionally for $X>1$: 
\begin{equation} \label{f_0315_2}
\aligned 
H_{1/2}(X)
& = -4 \sqrt{X} 
+ \sum_{n \leq X}\frac{\Lambda(n)}{\sqrt{n}} \log \frac{X}{n} 
+ \frac{\zeta'}{\zeta}\left(\frac{1}{2}\right)\log X \\
& \quad + 
\left(\frac{\zeta'}{\zeta}\right)^\prime\left(\frac{1}{2}\right)
- \frac{4}{\sqrt{X}} 
+  
\frac{1}{4\sqrt{X}} \, \Phi(X^{-2},2,1/4).
\endaligned 
\end{equation}
Furthermore, formula \eqref{f_0315_2} holds for $X=1$ 
in the sense that the right limit of the right-hand side at $X=1$ 
is equal to the value of the left-hand side at $X = 1$. 
\end{proposition}
\begin{proof}
By taking the limit $\ell \to 1/2$ on the right-hand side of \eqref{f_0315_1},  
\[
\aligned 
H_{1/2}(X)
& = -4 \sqrt{X} 
+ \sum_{n \leq X}\frac{\Lambda(n)}{\sqrt{n}} \log \frac{X}{n} 
+ \frac{\zeta'}{\zeta}\left(\frac{1}{2}\right)\log X + 
\left(\frac{\zeta'}{\zeta}\right)^\prime\left(\frac{1}{2}\right)\\
& \quad 
+\frac{1}{2\sqrt{X}} \lim_{\ell \to 1/2}
\frac{X^{-2}}{1-2\ell} \Bigl[ 
\Phi(X^{-2},1,1+\ell/2) \\ 
& \qquad \qquad \qquad \qquad \qquad \qquad - \Phi(X^{-2},1,1+(1-\ell)/2) \Bigr].
\endaligned 
\]
The second line on the right-hand side is equal to  
\[
\frac{1}{4\sqrt{X}} \cdot X^{-2}\Phi(X^{-2},2,1+1/4)
\] 
by l'H{\^o}pital's rule, because 
\[
\frac{\partial}{\partial a}\Phi(z,s,a) = -s \Phi(z,s+1,a)
\] 
is established by term-by-term differentiation of the series representation 
for $|z|<1$. Further, we have $X^{-2} \Phi(X^{-2},2,1+1/4) = -16 + \Phi(X^{-2},2,1/4)$ 
by applying 
\[
\Phi(z,s,a) = z^k \Phi(z,s,a+k) + \sum_{n=0}^{k-1} \frac{z^n}{(n+a)^s}
\]
(\cite[1.11 (2)]{Er81}) to $z=X^{-2}$, $s=2$, $a=1/4$, and $k=1$. 
Hence, we obtain \eqref{f_0315_2} for $X>1$. 
The right-hand side of \eqref{f_0315_2}  
is right continuous at $X=1$ by \cite[1.11 (3)]{Er81}, 
so the equation holds for $X=1$. 
\end{proof}

Proposition \ref{prop_6_3} gives an alternative proof of \cite[Theorem 1.1 (2)]{Su23}. 
In fact, we obtain 
\[
\aligned 
H_{1/2}(1)-H_{1/2}(e^t)
& = 4(e^{t/2} +e^{-t/2}-2) 
- \sideset{}{'}\sum_{n \leq e^t}\frac{\Lambda(n)}{\sqrt{n}} (t - \log n) \\
& \quad
-  \frac{\zeta'}{\zeta}\left(\frac{1}{2}\right) t
 + \frac{1}{4} \Bigl[ \Phi(1,2,1/4) - e^{-t/2} \Phi(e^{-2t},2,1/4) \Bigr], 
\endaligned 
\]
whose right-hand side coincides with the right-hand side of \cite[(1.1)]{Su23} 
because 
\[
\frac{\xi'}{\xi}\left(\frac{1}{2}\right)=
\frac{\zeta'}{\zeta}\left(\frac{1}{2}\right)
+ \frac{1}{2}\left[ \frac{\Gamma'}{\Gamma}\left(\frac{1}{4}\right) - \log \pi \right] = 0
\]
by the functional equation $\xi(1-s)=\xi(s):=s(s-1)\pi^{-s/2}\Gamma(s/2)\zeta(s)$. 
The left-hand side $H_{1/2}(1)-H_{1/2}(e^t)$ is equal to $\Psi(t)$ of \cite{Su23} 
by definition \eqref{f_109}, \cite[(1.3)]{Su23}, 
and  the symmetry of nontrivial zeros for $\rho \mapsto 1 -\rho$. 
Therefore, \cite[Theorem 1.1 (2)]{Su23} is proved. 

\section{On the equivalence of \eqref{f_102} and \eqref{f_106}}  \label{section_7}

We show the equivalence of \eqref{f_102} and \eqref{f_106} assuming the Riemann hypothesis.  
First, we derive \eqref{f_106} from \eqref{f_102}.
If we write $r_2(n)=\sum_{m+k=n} \Lambda(m)\Lambda(k)$ as in \cite{Fu1, Fu2,Fu3}, 
then 
\[
\aligned 
\sum_{n \leq X}\frac{r_2(n)}{n^2}
& = \frac{1}{X^2}\sum_{n \leq X}r_2(n)
+ 2 \int_{1}^{X} \left(\sum_{n \leq y}r_2(n) \right) y^{-3} \, dy \\
& = \frac{1}{2} - \frac{2}{\sqrt{X}} H(X) +\frac{1}{X^2}R(X) \\
& \quad 
+ 2 \int_{1}^{X} \left(\frac{1}{2}y^{-1}  - 2\, y^{-3/2} H(y) +y^{-3} R(y) \right) \, dy, 
\endaligned 
\]
by partial summation and \eqref{f_102}. 
The middle term of the integral on the right-hand side 
is calculated as 
\[
\int_{1}^{X}  y^{-3/2} H(y) \, dy 
 = 
\sum_\rho \frac{X^{\rho-1}-1}{\rho(\rho+1)(\rho-1)}
\]
by Fubini's theorem. 
For the sum on the right-hand side, we have
\[
\sum_\rho \frac{2X^{\rho-1}}{\rho(\rho+1)(\rho-1)} 
+ \frac{1}{\sqrt{X}}H(X)
= - \frac{1}{\sqrt{X}}H_1(X).
\]
On the other hand, the equality 
\[
\int_{1}^{X} y^{-3} R(y) \, dy
= \int_{1}^{\infty} y^{-3} R(y) \, dy - \int_{X}^{\infty} y^{-3} R(y) \, dy
\]
is justified under $R(y)=O(y^{1+\varepsilon})$, 
which follows from the Riemann hypothesis (\cite{BS10, LZ12}). 
By the above, we obtain
\[
\aligned 
\sum_{n \leq X}\frac{r_2(n)}{n^2}
& = \log X + \left(
\frac{1}{2} 
+\sum_\rho \frac{4}{\rho(\rho+1)(\rho-1)} 
+ 2 \int_{1}^{\infty}  y^{-3} R(y)  \, dy
\right) \\
& \quad 
+ \frac{2}{\sqrt{X}}\,H_1(X) +\frac{1}{X^2}R(X)
- 2 \int_{X}^{\infty}  y^{-3} R(y)  \, dy.
\endaligned 
\]
This gives \eqref{f_106} with 
\[
c_2 = \frac{1}{2} 
+\sum_\rho \frac{4}{\rho(\rho+1)(\rho-1)} 
+ 2 \int_{1}^{\infty}  y^{-3} R(y)  \, dy
\]
and 
\[
E(X) = \frac{1}{X^2}R(X)
- 2 \int_{X}^{\infty}  y^{-3} R(y)  \, dy.
\]
Then, \eqref{f_107} holds under the Riemann hypothesis, 
since $H_1(X)=O(1)$ and $E(X)=O(X^{-1+\varepsilon})$ by $R(y)=O(y^{1+\varepsilon})$.

Next, we derive \eqref{f_102} from \eqref{f_106}. We have 
\[
\aligned 
\sum_{n \leq X} r_2(n)
& = X^2\sum_{n \leq X}\frac{r_2(n)}{n^2}
- 2 \int_{1}^{X} \left(\sum_{n \leq y}\frac{r_2(n)}{n^2} \right) y \, dy \\
& = X^2\log X + c_2X^2 + 2X^{3/2} \, H_1(X) + X^2E(X) \\
& \quad 
- 2 \int_{1}^{X} \left(y\log y + c_2y + 2\sqrt{y} \, H_1(y) + yE(y) \right) \, dy
\endaligned 
\]
by partial summation and \eqref{f_106}. 
The third term of the integral on the right-hand side 
is calculated as 
\[
\aligned 
\int_{1}^{X}  \sqrt{y} H_1(y) \, dy 
&
 = -
\sum_\rho \frac{X^{\rho+1}-1}{\rho(\rho+1)(\rho-1)} 
\endaligned 
\]
by Fubini's theorem. For the sum on the right-hand side, we have
\[
\sum_\rho \frac{2X^{\rho+1}}{\rho(\rho+1)(\rho-1)}+ X^{3/2} \, H_1(X) 
= - X^{3/2}H(X). 
\]
On the other hand, 
\[
2 \int_{1}^{X} (y\log y + c_2y  ) dy = X^2\log X + (X^2-1)\left(c_2-\frac{1}{2}\right). 
\]
Therefore, we obtain
\[
\aligned 
\sum_{n \leq X} r_2(n)
& = \frac{1}{2}X^2  - 2X^{3/2} \, H(X) + X^2E(X) \\
& \quad 
- 2 \int_{1}^{X}  yE(y) \, dy+ c_2-\frac{1}{2}
- \sum_\rho \frac{4}{\rho(\rho+1)(\rho-1)} . 
\endaligned 
\]
This gives \eqref{f_102} with 
\[
R(X) = X^2E(X)
- 2 \int_{1}^{X}  yE(y) \, dy+ c_2-\frac{1}{2} - \sum_\rho \frac{4}{\rho(\rho+1)(\rho-1)} . 
\]
From this, the conjectural estimate $E(X)=O(X^{-1+\varepsilon})$ implies 
$R(X)=O(X^{1+\varepsilon})$, since 
the constants on the right-hand side are absorbed into other terms. 
\bigskip

\noindent
{\bf Acknowledgments}~
The authors thank the referee for careful reading of the manuscript and 
providing valuable comments, 
and Masahiro Mine for several helpful comments, including those on the support of $M$-functions in Theorem~\ref{thm_2_1}.
The first and second authors were supported by JSPS KAKENHI 
Grant Number JP22K03276 and JP23K03050, respectively. 
This work was also supported by the Research Institute for Mathematical Sciences,
an International Joint Usage/Research Center located in Kyoto University.

%

%
\bigskip 

\noindent
Kohji Matsumoto,\\[5pt]
Graduate School of Mathematics \\
Nagoya University \\
Chikusa-ku \\
Nagoya 464-8602, Japan \\[2pt]
and \\[2pt]
Center for General Education \\
Aichi Institute of Technology \\
1247 Yachigusa, Yakusa-cho \\
Toyota 470-0392, Japan  \\[2pt]
Email: {\tt kohjimat@math.nagoya-u.ac.jp}
\bigskip

\noindent
Masatoshi Suzuki,\\[5pt]
Department of Mathematics, 
School of Science, \\
Institute of Science Tokyo \\
2-12-1 Ookayama, Meguro-ku, \\
Tokyo 152-8551, Japan  \\[2pt]
Email: {\tt msuzuki@math.sci.isct.ac.jp}
\end{document}